\documentclass[11pt]{amsart}

\usepackage{fullpage}
\usepackage{xcolor,cancel}

\usepackage{hyperref}

\usepackage{amsmath,amssymb,amsfonts,mathabx,wrapfig,enumerate,multicol,tikz-cd,tikz,caption,bm,relsize} 
\usepackage{subcaption}
\usepackage[numbers]{natbib}

\tikzset{>=latex}

\newtheorem{theorem}{Theorem}[section]

\newtheorem{lemma}[theorem]{Lemma}
\newtheorem{conj}[theorem]{Conjecture}
\newtheorem{prop}[theorem]{Proposition}
\newtheorem{corollary}[theorem]{Corollary}

\newtheorem{definition}[theorem]{Definition}
\newtheorem{obs}[theorem]{Observation}
\newtheorem{innercustomthm}{Theorem}
\newenvironment{customthm}[1]
  {\renewcommand\theinnercustomthm{#1}\innercustomthm}
  {\endinnercustomthm}

\theoremstyle{definition}

\DeclareMathOperator{\Ric}{Ric}

\DeclareMathOperator{\2}{II}

\DeclareMathOperator{\Bee}{B}
\DeclareMathOperator{\Ee}{E}

\DeclareMathOperator{\Ex}{X}

\DeclareMathOperator{\En}{N}
\DeclareMathOperator{\Em}{M}
\DeclareMathOperator{\Sphere}{S}
\DeclareMathOperator{\RP}{\mathbf{R}P}
\DeclareMathOperator{\CP}{\mathbf{C}P}
\DeclareMathOperator{\HP}{\mathbf{H}P}
\DeclareMathOperator{\OP}{\mathbf{O}P}
\DeclareMathOperator{\Disk}{D}
\DeclareMathOperator{\Inertia}{\mathcal{I}}
\DeclareMathOperator{\Structure}{\mathcal{S}^{Top/O}}
\DeclareMathOperator{\pRc}{\mathcal{R}^\text{pRc}}
\DeclareMathOperator{\psc}{\mathcal{R}^\text{psc}}
\DeclareMathOperator{\MpRc}{\mathcal{M}^\text{pRc}}
\DeclareMathOperator{\Mpsc}{\mathcal{M}^\text{psc}}
\DeclareMathOperator{\Riem}{\mathcal{R}}
\DeclareMathOperator{\Cores}{\mathcal{R}^\text{Cores}}
\DeclareMathOperator{\MCores}{\mathcal{M}^\text{Cores}}
\DeclareMathOperator{\kap}{cap}
\DeclareMathOperator{\Diff}{Diff}
\DeclareMathOperator{\Class}{\mathbf{Cores}}

\title{The space of positive Ricci curvature metrics on spin manifolds}
\author{Bradley Lewis Burdick}

\begin{document}
\maketitle

\begin{abstract} In this note we show that the space of all metrics of positive Ricci curvature on a spin manifold of dimension $4k-1$ for $k\ge 2$ has infinitely many path components provided that the manifold admits a very particular kind of metric.  \end{abstract}


\section{Introduction}\label{intro} 

\subsection{Main Results}


Associated to any smooth manifold $\Em^n$ there is the space of Riemannian metrics $\Riem(\Em^n)$. Regardless of the topology of $\Em^n$, $\Riem(\Em^n)$ is always convex and hence contractible. If we instead restrict ourselves to the space $\psc(\Em^n)$ of metrics of positive scalar curvature (henceforth psc), the topology may be very complicated for certain $\Em^n$. For a closed spin manifold of dimension $4k$, there is an obstruction for $\psc(\Em^{4k})$ to even be nonempty, the $\hat{A}$-genus. It was observed in \cite{Carr} that the $\hat{A}$-genus can be used in certain settings to detect multiple path components of $\psc(\Em^{4k-1})$ for a spin boundary. By constructing psc metrics on exotic spheres in dimension $(4k-1)$, Carr was able to show that $\psc(\Sphere^{4k-1})$ has infinitely many path components. 

We may restrict ourselves further to the space $\pRc(\Em^n)$ of all metrics of positive Ricci curvature (henceforth pRc), and ask if $\pRc(\Sphere^{4k-1})$ has infinitely many path components as well. Using a technique for performing pRc surgery developed in \cite{WraithSurgery}, an analogous family of pRc metrics on exotic spheres in dimension $(4k-1)$ is constructed in \cite{WraithExotic}. In \cite{WraithSphere} it was shown that each of these pRc metrics lie in the same path component of $\psc(\Sphere^{4k-1})$ as the metrics constructed in \cite{Carr} and hence $\pRc(\Sphere^{4k-1})$ has infinitely many path components as well. 

Using the psc connected sum of \cite{GL}, it is possible to extend Carr's argument to show that $\psc(\Em^{4k-1})$ has infinitely many path components for \emph{any} spin manifold, provided it admits a single psc metric (see \cite[Theorem 4.2.2.2]{TW}). In our previous work in \cite{BLBOld, BLBThesis,BLBNew}, we explored a technique for constructing pRc connected sums introduced in \cite{Per1}. In this note we will describe how it is possible to combine this technique for constructing pRc connected sums with the work of \cite{WraithExotic} to expand the result of \cite{WraithSphere} to show that $\pRc(\Em^{4k-1})$ has infinitely many path components for any spin manifold that admits a very particular kind of pRc metric, which we call \emph{core metrics} (see Definition \ref{Core}).

The author's previous work in \cite{BLBOld,BLBThesis,BLBNew} is dedicated to constructing examples of core metrics, which we summarize below in Theorem \ref{TheCores} and Corollary \ref{Connect}. following class of manifolds all admit core metrics. 

\pagebreak

\begin{definition}\label{ClassDefinition} 
Let $\Class$ be the class of smooth manifolds that contains $\Sphere^n$, $\CP^n$, $\HP^n$, and $\OP^2$ for $n\ge 2$ and satisfies the following two conditions.
\begin{enumerate}
\item \textbf{Closed under sphere bundles:} If $\Em^n\in \Class$ and $\Ee\rightarrow \Em^n$ is a rank $m\ge 4$ vector bundle, then $\Sphere(\Ee)\in \Class$. 
\item \textbf{Closed under connected sum:} If $\Em_1^n,\Em_2^n\in \Class$, then $\Em_1^n\# \Em_2^n \in \Class$. 
\end{enumerate}
\end{definition}

\noindent Our main result can be stated as follows.  

\begin{customthm}{A}\label{A} For $k\ge 2$, let $\Em^{4k-1} \in \Class$ be spin, then $\pRc\left(\Em^{4k-1}\right)$ has infinitely many path components. Moreover, if $L^{4k-1}$ is any lens space, then $\pRc\left(L^{4k-1}\#\Em^{4k-1}\right)$ has infinitely many path components. 
\end{customthm} 

\noindent We caution the reader that while every element of $\Class$ admits a core metric, it need not be spin (e.g. $\CP^{2k}\times\Sphere^3$ ). The proof of Theorem \ref{A} relies on the technique used in \cite{Carr}, which requires the manifold to be spin for the $\hat{A}$-genus to be an obstruction to the existence of a psc metric. To our knowledge Theorem \ref{A} provides the first examples of non-simply connected manifolds in infinitely many dimensions for which the space of pRc metrics has infinitely many path components. 

While Definition \ref{ClassDefinition} collects the most explicit examples for which the conclusions of Theorem \ref{A} hold, the proof of Theorem \ref{A} is based on a more general principle inspired by the work of \cite{Per1}. This principle is that two pRc manifolds with isometric boundaries can be glued together if the principal curvatures of one dominate the negative of the other, which allows us to think of pRc connected sums in terms of classes of pRc metrics satisfying certain boundary conditions. We are able to show that the conclusions of Theorem \ref{A} hold for any spin manifold that admits a core metric or a slightly weaker class of metrics, which we call \emph{socket metric} (see Definition \ref{Socket}).

\begin{customthm}{B}\label{B} For $k\ge 2$, if $\Em^{4k-1}$ is spin then $\pRc(\Em^{4k-1})$ has infinitely many path components provided that $\Em^{4k-1}$ admits a socket metric. 
\end{customthm}

\noindent Roughly, one can think of a core metric as a pRc metric with an embedded round hemisphere and a socket metric as a pRc metric that \emph{almost} admits an embedding of a round hemisphere. While every core metric is a socket metric, by \cite[Theorem 1]{Law} any manifold that admits a core metric is simply connected. By stating Theorem \ref{B} for manifolds that admit socket metrics, we are able to include the lens spaces in Theorem \ref{A}.  

While the topology of $\psc(\Em^n)$ and $\pRc(\Em^n)$ is interesting in its own right, we typically do not consider two metrics as distinct if there is an isometry between them. The action of the diffeomorphism group of $\Em^n$ acts on $\psc(\Em^n)$ and $\pRc(\Em^n)$ by pull-back so that the orbits of this action are isometry classes of psc and pRc metrics. We therefore define the moduli spaces of psc metrics $\Mpsc(\Em^n)$ and the moduli space of pRc metrics $\MpRc(\Em^n)$ as the orbit spaces of $\psc(\Em^n)$ and $\pRc(\Em^n)$ respectively under the action of the diffeomorphism group. 


In \cite{KS}, it was shown that $\Mpsc(\Ex^{4k-1})$ has infinitely many path components for any spin manifold $\Ex^{4k-1}$ admitting a psc metric that satisfies certain topological conditions (explained below in Section \ref{pRcsums}). This is achieved by showing that none of the metrics constructed using \cite{Carr,GL} that lie in the infinitely many path components of $\psc(\Ex^{4k-1})$ can lie in the same path component of $\Mpsc(\Ex^{4k-1})$. As $\Sphere^{4k-1}$ automatically satisfies the topological conditions of \cite{KS}, in \cite{WraithSphere} it is shown that $\MpRc(\Sphere^{4k-1})$ has infinitely many path components. We similarly are able to claim that $\MpRc(\Em^{4k-1})$ has infinitely many path components for any of the manifolds listed in Theorem \ref{A}, but only if they satisfy the topological conditions of \cite{KS}. 

\begin{customthm}{C}\label{ModuliList} Let $\Em^{4k-1}$ be any manifold of $\Class$ built out of only $\Sphere^n$ and \emph{trivial} sphere bundles, then $ \MpRc(\Em^{4k-1})$ has infinitely many path components. 

 Let $m$ be an odd number and $q_1,\dots, q_{2k}$ be any numbers relatively prime to $m$ that satisfy $e_i(q_1^2,\dots, q_{2k}^2)\equiv 0 \mod m$ for $1\le i\le k$, where $e_i$ denotes the elementary symmetric polynomial in $2k$-variables. Then $\MpRc(L^{4k-1}\# \Em^{4k-1})$
have infinitely many path components, where $L^{4k-1}$ is the lens space $L(m;q_1,\dots, q_{2k})$. 
\end{customthm}

\noindent We caution the reader that the assumption on the lens space is required to ensure that the topological assumptions of \cite{KS} are satisfied. For example $L(3;q_1,q_2,q_3,q_4)$ is not included in Theorem \ref{ModuliList} for any choice of $q_i$, but $L(5;1,1,2,2)$ is. While we make no general attempt to solve for admissible $q_i$ for a given $m$, as explained in \cite[Theorem]{EMSS}, for each fixed $k$ there is a sufficiently large prime $m$ and a lens space of dimension $4k-1$ with fundamental group $\mathbf{Z}/m\mathbf{Z}$ for which the hypotheses of Theorem \ref{ModuliList} can be satisfied. Thus Theorem \ref{ModuliList} implies that, for each $k\ge 2$, there is a non-simply connected manifold in dimension $4k-1$ for which the moduli space of pRc metrics has infinitely many path components.

While we have listed the most concrete examples for which the conclusion of Theorem \ref{ModuliList} is known to hold, the proof holds for any manifold that satisfies the topological conditions of \cite{KS} and the metric conditions of Theorem \ref{B}. 

\begin{customthm}{D}\label{SocketModuli} For $k\ge 2$, if $\Em^{4k-1}$ is a spin manifold such that $H^1(\Em^{4k-1},\mathbf{Z}/2\mathbf{Z})=0$ and all Pontryagin classes vanish, then $\MpRc(\Em^{4k-1})$ has infinitely many path components provided that $\Em^{4k-1}$ admits a socket metric. 
\end{customthm}

\subsection{Outline} Theorem \ref{A} follows from Theorem \ref{B} and our knowledge of manifolds that admit socket metrics developed in \cite{BLBOld,BLBThesis,BLBNew}. This is explained in the proof of Theorem \ref{A} at the end of Section \ref{pRcsums}. Similarly, Theorem \ref{ModuliList} follows from Theorem \ref{SocketModuli} and checking which of those manifolds listed in Theorem \ref{A} satisfy the topological conditions of \cite{KS}. This is explained in the proof of Theorem \ref{B} at the end of Section \ref{pRcsums}. 

The proofs of Theorem \ref{B} and \ref{SocketModuli} rely on same elementary observation as the proof of \cite[Theorem A]{WraithSphere}, which is that $\pRc(\Em^n)\subseteq \psc(\Em^n)$ and $\MpRc(\Em^n)\subseteq \Mpsc(\Em^n)$. In order to prove Theorems \ref{B} and \ref{SocketModuli}, it therefore suffices to construct a sequence of pRc metrics that are psc-isotopic to psc metrics that are known to lie in distinct path components of $\psc(\Em^n)$ and $\Mpsc(\Em^n)$ respectively. In Section \ref{pscSection} we review the topological invariants used to detect these path components, and in Section \ref{pRcSection} we construct these sequence of pRc metrics. 

In Section \ref{pscSection}, we sketch the proofs of the analogous versions of Theorems \ref{B} and \ref{SocketModuli} for psc metrics included below as Theorems \ref{SpinTheorem} and \ref{KSTheorem} respectively. In Section \ref{CarrSection}, we describe how the work of \cite{Carr} uses the $\hat{A}$-genus to detect path components of $\psc(\Em^{4k-1})$ in Theorem \ref{SpinTheorem}. In Section \ref{KSSection}, we describe how the work of \cite{KS} uses their $s$-invariant to detect path components of $\Mpsc(\Em^{4k-1})$ in Theorem \ref{KSTheorem}. We include enough detail for the reader to easily understand why the pRc metrics constructed in Section \ref{pRcSection} lie in distinct path components of $\psc(\Em^{4k-1})$ and $\Mpsc(\Em^{4k-1})$.

In Section \ref{pRcSection}, we carry out the construction of pRc metrics and ultimately prove Theorems \ref{B} and \ref{SocketModuli}. We begin in Section \ref{WraithSection} by explaining how the work of \cite{WraithSphere} was able to generalize the work of \cite{Carr} to pRc metrics. The main results of \cite{WraithSphere}, included as Theorems \ref{WraithSpaceTheorem} and \ref{WraithModuliTheorem}, only discuss the nontriviality of $\pRc(\Sphere^{4k-1})$ and $\MpRc(\Sphere^{4k-1})$. But a close reading of the details shows that the arguments applies equally well to any spin manifold admitting a particular family of pRc metrics, which we explain below in Theorems \ref{surgerytheorem} and \ref{SpinModuli}. We emphasize that Theorems \ref{B} and \ref{SocketModuli} are essentially a reframing of these two theorems in terms that are compatible with our previous work. In Section \ref{pRcsums}, we review concepts introduced in our previous work in \cite{BLBOld,BLBNew,BLBThesis}, specifically we discuss what exactly we mean by core and socket metrics. We also summarize and rephrase results needed for our application in this note. Finally in Section \ref{SurgeryCores} we make the connection between the work of \cite{WraithSphere} we summarized in Section \ref{WraithSection} and our own ideas presented in Section \ref{pRcsums} in the form of Lemma \ref{DiskLemma}, which immediately implies Theorems \ref{B} and \ref{SocketModuli}. Lemma \ref{DiskLemma} claims there is a family of metrics on the disk that simultaneously satisfy the metric hypotheses of Theorems \ref{surgerytheorem} and \ref{SpinModuli} and that satisfy the boundary conditions discussed in Section \ref{pRcsums} that allow them to be attached to any manifold admitting a socket metric. 

We end this note with Section \ref{AppsSection}, which gives two further applications of the techniques developed in this paper. In Section \ref{CoresSection} we discuss what is known about the topology of the space of all core metrics, present a conjecture about its structure, and end with Corollary \ref{CoreSpace}, which claims the space of core metrics will itself have infinitely many path components for a spin manifold of the appropriate dimension. And in Section \ref{ExoticSection} we remark that many of the pRc metrics constructed in this note will exist on exotic smooth structures. 


\section{The psc Story}\label{pscSection}



\subsection{The work of Carr}\label{CarrSection}


For a $4k$-dimensional Riemannian spin manifold $\Em^{4k}$ there is an invariant $\hat{A}(\Em^{4k}) \in \mathbf{Z}$ defined in terms of the index of the associated spinor Dirac operator. When $\Em^{4k}$ has positive scalar curvature, by \cite{Lich} the spinor Dirac operator is invertible and hence $\hat{A}(\Em^{4k})=0$. In \cite{AS}, it was shown that $\hat{A}(\Em^{4k})$ is computable in terms of the Pontryagin numbers of a closed spin manifold $\Em^{4k}$, and hence there is a topological obstruction to the existence of a psc metric on a closed spin manifold of dimension $4k$. 

The following observation is the key idea behind the main result in \cite{Carr}. 

\begin{obs}\label{Ahat} For a manifold $\Em^{4k-1}$ that bounds two spin manifolds $\Ex_1^{4k}$ and $\Ex_2^{4k}$, suppose that two psc metrics $g_1,g_2\in \psc(\Em^{4k-1})$ extend to two psc metrics $G_1$ on $\Ex_1^{4k}$ and $G_2$ on $\Ex_2^{4k}$ that each split as a product on a neighborhood of the boundary. Define the closed spin manifold $\Ex^{4k}=\Ex_1^{4k}\cup_{\Em^{4k-1}} \Ex_2^{4k}$. If $g_1$ and $g_2$ lie in the same path component of $\psc(\Em^{4k-1})$, then $\hat{A}(\Ex^{4k})=0$. 
\end{obs}

The proof of Observation \ref{Ahat} is straightforward. Suppose that $g_1$ and $g_2$ are connected via a path in $\psc(\Em^{4k-1})$. Then by \cite[Theorem]{Gaj}, there is a psc metric $H$ on $[0,1]\times \Em^{4k-1}$ that can be smoothly glued together with $G_1$ and $G_2$ to form a psc metric on $\Ex^{4k}$. Hence $\hat{A}(\Ex^{4k})=0$. Note that we can use the contrapositive of Observation \ref{Ahat} to detect distinct path components of $\psc(\Em^{4k-1})$ whenever we have $\hat{A}(\Ex^{4k})\neq 0$. 

The simplest spin manifold we may hope to apply Observation \ref{Ahat} to is $\Sphere^{4k-1}$. For $k\ge 2$, in these dimensions there is never a unique smooth structure on the sphere, first observed in \cite{KM}. The set of diffeomorphism classes of smooth structures on $\Sphere^n$ is denoted as $\Theta_n$. It is an abelian group with addition being given by connected sum. An important subgroup is given by those smooth structures that exist as the boundary of a smooth parallelizable $(n+1)$-dimensional manifold, denoted by $bP_{n+1}$. In particular, $bP_{4k}$ is cyclic of order $b_k$ (the exact value of $b_k$ can be found following \cite[Corollary 7.6]{KM}), generated by an element $\Sigma^{4k-1}$ that bounds a parallelizable manifold $\Ee_{8}^{4k}$ with signature $\sigma(\Ee_8^{4k})=8$. Each element of $bP_{4k}$ can be written as $\#_p \Sigma^{4k-1}$ and hence bounds $\Ee_{8p}^{4k}:=\natural_p \Ee_8^{4k}$. 

Noting that $\Sphere^{4k-1}\cong \#_p \Sigma^{4k-1}$ whenever $p=qb_k$, we have an infinite sequence of distinct spin manifolds $\Ee_{8qb_k}^{4k}$ each with boundary diffeomorphic $\Sphere^{4k-1}$. It is natural to ask if we can apply Observation \ref{Ahat} to this scenario. We first need a technique for constructing metrics that satisfy the hypotheses of Observation \ref{Ahat}. 

\begin{theorem}\label{CarrMetric}\cite[Theorem 3]{Carr} Let $K$ is a codimension $k\ge 3$ subcomplex of a smooth triangulation of a Riemannian manifold $M$ and let $U\subseteq \Em^n$ be a normal neighborhood of $K$, then $U$ caries a metric of positive scalar curvature which is a product near the boundary. 
\end{theorem} 

\noindent In order to understand how \ref{CarrMetric} applies to $\Ee_{8p}^{4k}$, we must briefly remark on their construction. Using a technique known as \emph{plumbing}, introduced by Milnor, $\Ee_{8p}^{4k}$ is built out of $\Disk^{2k}$-bundles over $\Sphere^{2k}$ (see \cite[Section V]{Browd} for a discussion) in such a way that it is obviously diffeomorphic to a normal neighborhood of its $(2k)$-skeleton. Hence for $k\ge 2$ we are able to make the following definition

\begin{definition}\label{CarrDefinition} For $k\ge 2$ and each $p\in \mathbf{Z}$ let $G_p^C$ denote the psc metric on $\Ee_{8p}^{4k}$ constructed in  Theorem \ref{CarrMetric}, and let $g_p^C$ denote its restriction to its boundary $\#_p \Sigma^{4k-1}$. 
\end{definition} 

\noindent With Definition \ref{CarrDefinition} in hand we may now explain how \cite{Carr} uses Observation \ref{Ahat} to prove the following. 

\begin{theorem}\label{CarrTheorem}\cite[Theorem 4]{Carr}\footnote{In actuality \cite[Theorem 4]{Carr} is only stated for the standard $\Sphere^{4k-1}$, but the argument is identical for each $\#_p\Sigma^{4k-1}$} Let $\Sigma^{4k-1}\in bP_{4k}$ be a generator, then $\psc(\#_p\Sigma^{4k-1})$ has infinitely many path components for each $p$. 
\end{theorem}

Note that $g_{p+qb_k}^C\in \psc(\#_p\Sigma^{4k-1})$ for each $q\in\mathbf{Z}$ and each extend to psc metrics $G_{p+qb_k}^C$ on $\Ee_{8(p+qb_k)}^{4k}$ as in Observation \ref{Ahat}. Let $q,q'$ be any two distinct integers. Define the closed spin manifold
$$\Ex_{q,q'}^{4k} = \Ee_{8(p+qb_k)}^{4k} \cup_{\#_p \Sigma^{4k-1}} \Ee_{8(p+q'b_k)}^{4k}.$$
By \cite[Theorem 12]{Carr}, we have $\hat{A}(\Ex_{q,q'}^{4k})=c(q-q')$ for some constant $c$. It follows from Observation \ref{Ahat} that $g_{p+qb_k}^C$ and $g_{p+q'b_k}^C$ lie in distinct path components for each $q\neq q'$, and hence $\psc(\#_p \Sigma^{4k-1})$ has infinitely many path components. 

Theorem \ref{CarrTheorem} was generalized to $\pRc(\#_p\Sigma^{4k-1})$, the space of all pRc metrics, in \cite{WraithSphere}, which we will review in Section \ref{WraithSection}. In Theorem \ref{B} we are claiming a generalization the following result. 

\begin{theorem}\cite[Theorem 4.2.2.2]{TW}\label{SpinTheorem} If $\Em^{4k-1}$ is a spin manifold that admits a psc metric, then $\psc(\Em^{4k-1})$ has infinitely many path components.
\end{theorem}

The proof relies on the psc connected sum operation of \cite[Theorem A]{GL} and the extension of this to boundary connected sums in \cite[Theorem]{Gaj}. 

\begin{theorem}\cite[Theorem]{Gaj}\label{GL} Let $h_1$ and $h_2$ be any psc metrics on $\Em^{n-1}_1$ and $\Em^{n-1}_2$ that extend to psc metrics $H_1$ on $\Ex^n_1$ and $H_2$ on $\Ex^n_2$ that are products on a neighborhood of the boundary, then there is a psc metric $h_1\#_{GL} h_2$ on $\Em_1^{n-1}\#\Em_2^{n-1}$ that extends to a psc metric $H_1\natural_{GL} H_2$ on $\Ex^n_1\natural \Ex_2^n$ that is a product on a neighborhood of the boundary. 
\end{theorem}

With Theorem \ref{GL} in hand, we can now explain how to prove Theorem \ref{SpinTheorem}. It uses an argument similar to Observation \ref{Ahat}, but modified as we cannot assume that $\Em^{4k-1}$ is a spin boundary. Let $h$ denote the hypothesized psc metric on $\Em^{4k-1}$ and let $H=dt^2+h$ on $[0,1]\times \Em^{4k-1}$, then by Theorem \ref{GL} there is a psc metric $H\natural_{GL} G_{qb_k}^C$ on $ \Ex_{q}^{4k} =  \left( [0,1]\times \Em^{4k-1}\right) \natural \Ee_{8(qb_k)}^{4k}$ that is isometric to $h\#_{GL}g_{qb_k}^C$ on $\{1\}\times \Em^ {4k-1}$ and isometric to $h$ on $\{0\}\times \Em^{4k-1}$. If $h\#_{GL}g_{qb_k}^C$ and $h\#_{GL} g_{q'b_k}^C$ were connected via a path of psc metrics, then by \cite[Theorem]{Gaj} there is a psc metric $K$ on $[0,1]\times \Em^{4k-1}$ that could be glued to $\Ex_q^{4k}$ with the metric $H\natural_{GL} G_{qb_k}^C$ on $\{1\}\times \Em^{4k-1}$ at one end and to $\Ex_{q'}^{4k}$ $H\natural_{GL} G_{q'b_k}^C$ on $\{1\}\times \Em^{4k-1}$ at the other end. The result is a psc metric on $ \left([0,1]\times \Em^{4k-1}\right) \# \Ex_{q,q'}^{4k}$ that restricted to either boundary is isometric to $h$. By identifying the two ends we have a psc metric on $\left(\Sphere^1\times \Em^{4k-1}\right)\# \Ex_{q,q'}^{4k}$. As the $\hat{A}$-genus is spin-bordism invariant we have
$$ \hat{A}\left( \left(\Sphere^1\times \Em^{4k-1} \right) \# W_{q,q'}^{4k}\right)=\hat{A}(\Sphere^1\times \Em^{4k-1})+ \hat{A}(W_{q,q'}^{4k})= c(q-q').$$
Where we have used the fact that $\hat{A}(\Sphere^1\times \Em^{4k-1})=0$ because $d\theta^2+h$ is a psc metric on $\Sphere^1\times \Em^{4k-1}$. We conclude that $h\#_{GL}g_{qb_k}^C$ and $h\#_{GL}g_{q'b_k}^C$ must lie in distinct path components, and hence $\psc(\Em^{4k-1})$ has infinitely many path components.


\subsection{The work of Kreck and Stolz}\label{KSSection}


One of the principal techniques in probing the moduli space $\Mpsc(\Em^{4k-1})$, pioneered in \cite{KS}, is the use of a particular analytically defined invariant $s(\Em^{4k-1},g) \in \mathbf{Q}$ defined for a psc metric $g$ on a spin manifold $\Em^{4k-1}$ with vanishing Pontryagin classes. The absolute value of this invariant detects whether two psc metrics lie in the same connected components of the orbit space of $\psc(\Em^{4k-1})$ under the action of \emph{spin} isometries. Under the assumption of a unique spin structure (that $H^1(\Em^{4k-1},\mathbf{Z}/2\mathbf{Z})=0$), $|s(\Em^{4k-1},g)|$ is a legitimate invariant of the connected components of $\Mpsc(\Em^{4k-1})$. Furthermore, when $\Em^{4k-1}$ is the boundary of a spin manifold $\Ex^{4k}$ and the psc metric extends to a psc metric $G$ on $\Ex^{4k}$ with product boundary, $s(\Em^{4k-1},g)$ can be computed in terms of the topology of $\Ex^{4k}$ \cite[Proposition 2.13]{KS}. Putting all of this together, the authors of \cite{KS} were able to prove the following about $\Mpsc(\Em^{4k-1})$. 

\begin{theorem}\cite[Corollary 2.15]{KS}\label{KSTheorem} If $\Em^{4k-1}$ is a spin manifold such that $H^1(\Em^{4k-1},\mathbf{Z}/2\mathbf{Z})=0$ and all Pontryagin classes vanish, then $\Mpsc(\Em^{4k-1})$ has infinitely many path components provided it is nonempty. 
\end{theorem}

The proof of Theorem \ref{KSTheorem} relies on computing the value of $|s|$ for the metrics $h\#_{GL} g_{qb_k}^C$ that lie in distinct path components of $\psc(\Em^{4k-1})$ as described in the proof outlined for Theorem \ref{SpinTheorem}. By \cite[Proposition 2.13 (iv)]{KS}, $s(\Em^{4k-1},h\#_{GL}g_{qb_k}^C) = s(\Em^{4k-1},h)+s(\Sphere^{4k-1},g_{qb_k}^C)$, and hence it suffices to show that $|s(\Sphere^{4k-1}, g_{qb_k}^C)|\neq |s(\Sphere^{4k-1},g_{q'b_k}^C)|$ for $q\neq \pm q'$. By \cite[Proposition 2.13]{KS} this reduces to a topological computation on $\Ee_{8qb_k}^{4k}$, which shows that $|s(\Sphere^{4k-1},g_{qb_k}^C)|= c q$ for some constant $c$. 


\section{The pRc Story}\label{pRcSection}



\subsection{The work of Wraith}\label{WraithSection}


In \cite{WraithSphere}, the author's previous constructions of pRc metrics on exotic spheres in \cite{WraithExotic} are used to prove the following generalization of Theorem \ref{CarrTheorem}. 

\begin{theorem}\cite[Theorem A]{WraithSphere}\label{WraithSpaceTheorem} Let $\Sigma^{4k-1}\in bP_{4k}$ be a generator, then $\pRc(\#_p \Sigma^{4k-1})$ has infinitely many path components. 
\end{theorem}

The idea behind the proof of Theorem \ref{WraithSpaceTheorem} is a construction of pRc metric on $\#_p \Sigma^{4k-1} =\#_{p+qb_k} \Sigma^{4k-1}$ for each $q$. While the construction of these metrics does not extend easily to metrics on $\Ee_{8(p+qb_k)}^{4k}$ with positive scalar curvature, they are constructed out of building blocks that respect this topology. It is therefore reasonable to suspect these metrics lie in distinct path components of $\pRc(\#_p\Sigma^{4k-1})$ for each $q$. As $\pRc(\Em^n)\subseteq \psc(\Em^n)$, it suffices to show that these metrics lie in distinct path components of $\psc(\#_p\Sigma^{4k-1})$. The proof of Theorem \ref{WraithSpaceTheorem} is by carefully showing that the pRc metrics are connected via a path in $\psc(\#_p\Sigma^{4k-1})$ to metrics similar to those in Definition \ref{CarrDefinition}. Hence the argument outlined above for Theorem \ref{CarrTheorem} also implies Theorem \ref{WraithSpaceTheorem}.

For details of the metric construction we recommend reading \cite{WraithSurgery,WraithExotic,WraithSphere}. We summarize the results we need in the following theorem. 

\begin{theorem}\cite[Proposition 5.5]{WraithSphere}\label{WraithMetric} For $k\ge 2$, suppose that $\Em^{4k-1}$ admits a pRc metric $h$ such that there is a nullhomotopic isometric embedding $\iota:\Sphere^{2k-1}_\rho\times \Disk^{2k}_R(N)\hookrightarrow \Em^{4k-1}$. For all $p$, if $\rho< \kappa(k,p,N,R)$ it is possible to perform iterated surgeries on $\iota$ to produce a pRc metric $h\#_W g_p^W$ on $\Em^{4k-1}\# \left(\#_p \Sigma^{4k-1}\right)$. 

Suppose moreover, there is a psc metric $H$ on $W^{4k}$ that splits as a product metric $dt^2+h$ near the boundary $\partial W^{4k}=\Em^{4k-1}$. Then the metric $h\#_W g_p^W$ is psc isotopic to a metric $h\#_Wg_{p,\infty}^W$ that extends to a psc metric $H\natural_W G_{p,\infty}^W$ on $W^{4k}\natural \Ee_{8p}^{4k}$ that splits as a  product metric $dt^2+ h\#_W g_{p,\infty}^W$ near the boundary $\Em^{4k-1}\# \left(\#_p \Sigma^{4k-1}\right)$. 
\end{theorem}

\begin{proof} As explained in \cite[Section 2]{WraithSphere}, it is possible to construct any $\#_p \Sigma^{4k-1}$ by performing iterated surgeries on $\Sphere^{2k-1}\times \Disk^{2k}$ starting with a fiber sphere of $\Sphere(T\Sphere^{2k})$. Note that
 $$\Sphere(T\Sphere^{2k}) = \left(\Disk^{2k}\times \Sphere^{2k-1}\right) \cup_F \left(\Disk^{2k}\times \Sphere^{2k-1}\right),$$
 where $F:\Sphere^{2k-1}\times \Sphere^{2k-1} \rightarrow \Sphere^{2k-1}\times \Sphere^{2k-1}$ is of the form $F(x,y)=(x,f(x)\cdot y)$ for $f\in \pi_{2k-1}(SO(2k))$.  We claim that it is possible to take a connected sum with $\Sphere(T\Sphere^{2k})$ by performing framed surgery on $\iota$. 
\begin{equation}\label{1surgery} \left[\Em^n\setminus \iota \left( \Sphere^{2k-1}\times \Disk^{2k} \right)\right] \cup_{\iota\circ F}\left[ \Disk^{2k}\times \Sphere^{2k-1}\right] = \Em^n \# \Sphere(T\Sphere^{2k}) .\end{equation}
This follows from the fact that a nullhomotopic embedding factors through an embedding $\Disk^n\hookrightarrow \Em^n$ and the following elementary observation:
$$ \Disk^n \setminus \left(\Sphere^{2k-1}\times \Disk^{2k}\right) = \left(\Disk^{2k}\times \Sphere^{2k-1}\right) \setminus  \Disk^n .$$

We can now construct $\Em^{4k-1}\#\left( \#_p \Sigma^{4k-1}\right)$ by performing iterated surgeries on the fiber sphere of $\Sphere(T\Sphere^{2k})$ on the right-hand side of (\ref{1surgery}). If we fix $p$, the discussion following  \cite[Theorem 3.3]{WraithSphere} explains that $\Em^n\# \left(\#_p\Sigma^n\right)$ can be endowed with a pRc metric if $\rho$ is chosen sufficiently small depending on $R$, $N$, $k$, and $p$. This establishes the first claim. The second claim follows from \cite[Proposition 5.5]{WraithSphere}. 
\end{proof}

\noindent While the second paragraph of Theorem \ref{WraithMetric} follows directly from the work of \cite[Proposition 5.5]{WraithSphere}, it is important to note that the proof of this fact occupies the bulk of the paper. 

It is worth noting Theorem \ref{WraithMetric} combined with the logic of Theorem \ref{SpinTheorem} goes a great deal further than what is explicitly claimed in \cite{WraithSphere}. 

\begin{theorem}\label{surgerytheorem} For $k\ge 2$, suppose for a fixed $R,N>0$ and for any $\rho>0$ that $\Em^{4k-1}$ admits a of pRc metrics $h$ and a nullhomotopic isometric embedding $\iota:\Sphere^{2k-1}_\rho\times \Disk^{2k}_R(N)\hookrightarrow \Em^n $. If $\Em^{4k-1}$ is spin, then $\pRc(\Em^{4k-1})$ has infinitely many path components. 
\end{theorem}

\begin{proof} For a fixed $m>0$, there is a $\kappa(k,m,R,N)$ as in Theorem \ref{WraithMetric}, such that, if $\rho < \kappa$ we may find a pRc metric $h \#_W g_{qb_k}^W$ on $\Em^{4k-1}$ for all $|qb_k|\le m$. We claim that these metrics each lie in distinct path components. This shows that $\pRc(\Em^{4k-1})$ has more than $\lfloor m/b_k\rfloor$ path components for each $m$, and hence has an infinite number of path components. 

Setting $W^{4k} = [0,1]\times \Em^{4k-1}$ with $H= dt^2+h$, we may apply the second claim of Theorem \ref{WraithMetric} to find a psc metric $H\natural_W G_{p,\infty}^W$ on $W^{4k}\natural \Ee_{8(qb_k)}^{4k}$ that is a product near the boundary. The argument now proceeds identically to the proof of Theorem \ref{SpinTheorem} outlined above to show that $h\#_W g_{qb_k,\infty}^W$ lie in distinct path components of $\psc(\Em^{4k-1})$ for each $|q|\le \lfloor m/b_k\rfloor$. 
\end{proof}

In order to prove Theorem \ref{WraithSpaceTheorem} using Theorem \ref{WraithMetric}, it suffices to construct a metric on $\Sphere^{4k-1}$ that satisfies the hypotheses \emph{for a fixed} $R$ and $N$ and \emph{for all choices of }$\rho>0$. This can be achieved as follows. If we take the Riemannian product $\Sphere_\rho^{2k}\times \Sphere_1^{2k-1}$, there are two disjoint embeddings embeddings $\iota, \iota': \Sphere_\rho^{2k} \times \Disk_1^{2k-1}(\pi/4)\hookrightarrow \Sphere_\rho^{2k}\times \Sphere_1^{2k-1}$. By \cite[Lemma]{SY}, you can perform surgery on $\iota'$ to produce a pRc metric on $\Sphere^{2k-1}$ with an isometric embedding $\iota: \Sphere_\rho^{2k} \times \Disk_1^{2k-1}(\pi/4) \hookrightarrow \Sphere^{4k-1}$. This procedure can be repeated for any $\rho>0$, hence the second statement of Theorem \ref{WraithMetric} now implies Theorem \ref{WraithSpaceTheorem} as outlined above. While this is \emph{not} how Theorem \ref{WraithSpaceTheorem} is proven in \cite{WraithSphere}, it is illustrative of the perspective we wish to take in the proof of Lemma \ref{DiskLemma} in Section \ref{SurgeryCores}. 

While Theorem \ref{surgerytheorem} is essentially obvious from reading \cite{WraithSphere}, the main issue with it is that there are very few obvious candidates for $\Em^n$ that satisfy the hypotheses. The following Proposition is an incomplete list of those manifolds already known to in the literature to admit a family of metrics as in Theorem \ref{surgerytheorem}.

\begin{prop}\label{list} For $n=4k-1>3$, $\pRc(\Em^n)$ has infinitely many path components, where $\Em^n$ is any of the following
\begin{enumerate}
\item\label{product}\cite{SY,Ehr} Where $\Em^n$ is the result of performing surgery on the fiber of a linear $\Sphere^{2k-1}$-bundle over $\Ex^{2k}$, where $\Ex^{2k}$ is a spin manifold admitting a metric of pRc
\item\cite{WraithExotic} $\Em^n\in bP_{4k}$, 
\item\cite{CW} $\Em^n\#\Sigma^n$, where $\Em^n$ is any $(2k-2)$-connected, $(2k-1)$-parallelizable manifold and $\Sigma^n$ is some $\Sigma^n\in \Theta_n$. 
\end{enumerate}
\end{prop}

Note that each of the manifolds listed in Proposition \ref{list} are built in such a way that they are locally equivalent to a linear $\Sphere^{2k-1}$-bundle over $\Sphere^{2k}$. It is the basic fact that these bundles may be equipped with metrics locally isometric to a product, and that by shrinking the spherical fiber the Ricci curvatures converge to the corresponding Ricci curvatures in the product metric \cite[Proposition 9.70]{Besse}. Hence they can all be equipped, at least locally, with a family of pRc metrics
 as in Theorem \ref{surgerytheorem}. 

The list in Proposition \ref{list} is already quite impressive, particularly (\ref{product}) which only requires a spin manifold of dimension $2k$ that admits a metric of pRc. That said, the way that Theorem \ref{surgerytheorem} is phrased makes it difficult to apply to a manifold that we do not already know something of the global topology. That is unless we know upfront that $\Em^n$ can be constructed out of sphere bundles. Our main contribution to this subject, which we will outline in the next section, is that instead of knowing global topology of $\Em^n$ we know something of the global geometry of $\Em^n$.

Note that $\Em^{4k-1}=\#_p \Sigma^{4k-1}$ itself satisfies the topological hypotheses of Theorem \ref{KSTheorem}, hence $\Mpsc(\Sigma^{4k-1})$ has infinitely many path components. It is reasonable to ask if the metrics of pRc of Theorem \ref{WraithSpaceTheorem} lie in distinct path components of $\MpRc(\Sigma^{4k-1})$. 

\begin{theorem}\cite[Theorem A]{WraithSphere}\label{WraithModuliTheorem} For $k\ge 2$, let $\Sigma^{4k-1}\in bP_{4k}$ be a generator, then the moduli space $\MpRc(\#_p\Sigma^{4k-1})$ has infinitely many path components.
\end{theorem}

The proof of Theorem \ref{WraithModuliTheorem} is almost immediate from the proof of Theorem \ref{KSTheorem} and Theorem \ref{WraithSpaceTheorem} because \cite[Proposition 2.13]{KS} works equally well for \emph{any} psc metric on $\#_p \Sigma^{4k-1}$ that extends to a psc metric on $\Ee_{8(p+qb_k)}^{4k}$ that is a product near the boundary. We immediately see that $|s(\#_p\Sigma^{4k-1},g_{(p+qb_k),\infty}^W|=c(p+qb_k)$, and hence $g_{(p+qb_k),\infty}^W$ lie in distinct path components of $\MpRc(\#_p\Sigma^{4k-1})$ for all $q\neq \pm q'$. Here $g_{p,\infty}^W$ are those metrics constructed in Theorem \ref{WraithMetric} that are psc isotopic to the pRc metric $g_{(p+qb_k)}^W$. Hence $g_{(p+qb_k)}^W$ lie in distinct path components of $\MpRc(\#_p\Sigma^{4k-1})$ for $q\neq \pm q'$. 

As Theorem \ref{KSTheorem} is true for any spin manifold satisfying the topological hypotheses, it is reasonable to ask if the same result will hold for $\MpRc(\Em^{4k-1})$ for any spin manifolds satisfying both the topological conditions of Theorem \ref{KSTheorem} and the metric conditions of Theorem \ref{surgerytheorem}. There is only one slight subtlety that prevents the argument of Theorem \ref{KSTheorem} from applying automatically, and that is that \cite[Proposition 2.13 (iv)]{KS} relies on \cite[Theorem]{Gaj} to construct a psc metric on the manifold 
$$X^{4k}= \left([0,1]\times \Em^{4k-1} \right) \natural \left([0,1]\times \Sphere^{4k-1}\right),$$
that is isometric to $h\sqcup g_{qb_k}^C$ at one end and $h\#_{GL} g_{qb_k}^C$ at the other end. While we suspect that the pRc metric $h\#_W g_{qb_k}^W$ constructed in Theorem \ref{WraithMetric} is psc isotopic to a metric that can be extended over $X^{4k}$, this is not claimed in \cite{WraithSphere} and proving this would require us to retread all of the fine details of \cite{WraithSphere}. One way to side-step this issue is to instead consider the psc metric $H\natural_W G_{qb_k,\infty}^W$ constructed in Theorem \ref{WraithMetric} on the manifold
$$W^{4k}_q = \left([0,1]\times \Em^{4k-1}\right) \natural \left( \Ee_{8(qb_k)}^{4k}\right).$$

\begin{theorem}\label{SpinModuli} For $k\ge 2$, let $\Em^{4k-1}$ be a spin manifold satisfying the topological conditions of Theorem \ref{KSTheorem} and the metric conditions of Theorem \ref{surgerytheorem}. Then $\MpRc(\Em^{4k-1})$ has infinitely many path components. 
\end{theorem}

\begin{proof} As we have already pointed out, there is pRc metric $h\#_Wg_{qb_k}^W$ on $\Em^{4k-1}$ that is psc isotopic to a metric $h\#_Wg_{qb_k,\infty}^W$ that extends across the manifold $W_q^{4k}$. We claim that $|s(\Em^{4k-1},h\#_W g_{qb_k}^W) |= |s(\Em^{4k-1},h\#_W g_{qb_k,\infty}^W)|$ is distinct for every $q\in \mathbf{N}$. Because $W_q^{4k}$ admits a psc metric, \cite[Remark 2.2 (ii) \& Proposition 2.13 (iii)]{KS} implies that
$$ s(\Em^{4k-1},h\#_W g_{qb_k,\infty}^W) - s(\Em^{4k-1},h) = s(\Em^{4k-1}\sqcup (-\Em^{4k-1}), (h\#_W g_{qb_k,\infty}^W ) \sqcup h) = t(W_q^{4k}),$$
where a formula for $t(W_q^{4k})$ is found in \cite[Equation (2.11)]{KS}. 

We claim that the Pontryagin classes of $W_q^{4k}$ all vanish. First note that $\Ee_{8qb_k}^{4k}$ is parallelizable by construction, and hence $p_i(\Ee_{8qb_k}^{4k})=0$. Note that $T([0,1]\times \Em^{4k-1})= \varepsilon^1\oplus \pi^*T\Em^{4k-1}$, where $\varepsilon^1$ is the trivial line bundle and $\pi:[0,1]\times \Em^{4k-1}\rightarrow \Em^{4k-1}$ is the projection. Using the integral Whitney sum formula for Pontryagin classes \cite[Theorem 1.6]{Brown} we see that 
$$p_i\left(\varepsilon^1\oplus \pi^*T\Em^{4k-1}\right)=p_i(\pi^*T\Em^{4k-1})=\pi^*p_i(\Em^{4k-1})=0.$$
Finally we note that $p_i(X\natural Y) = p_i(X)+p_i(Y)$ for any two manifolds with boundary, and hence $p_i(W_q^{4k})=0$. 

Because all the Pontyragin classes of $W_q^{4k}$ vanish, \cite[Equation (2.11)]{KS} reduces to
$$t(W_p^{4k}) = c\sigma(W_q^{4k}),$$
for some constant $c$. Note that $\sigma(W_q^{4k}) = \sigma(\Ee_{8qb_k}^{4k}) = 8qb_k$. We conclude that 
$$|s(\Em^{4k-1},h\#_W g_{qb_k,\infty}^W)|= |s(\Em^{4k-1},h) + cq|,$$
and hence are distinct for each $q\in \mathbf{N}$. 
\end{proof}


\subsection{The author's previous work}\label{pRcsums}


The main goal of the author's previous papers \cite{BLBOld,BLBNew} and dissertation \cite{BLBThesis} was to construct new examples of pRc connected sums. While \cite{GL} implies that the connected sum of any two psc Riemannian manifolds has a metric of positive scalar curvature, such a result is not possible for pRc metrics. While no general result is possible, there have been many examples of pRc connected sums constructed in the literature \cite{Chee,SY2,SY,SY3,Per1,WraithExotic,WraithNew,SW,CW}. Our original goal in writing \cite{BLBOld} was generalizing the result in \cite{Per1}, which is a construction of pRc metric on $\#_k\CP^2$, to the connected sum of any projective spaces. There were two techniques presented in \cite{Per1} that reduce the problem to constructing a pRc metric on $\CP^2\setminus \Disk^4$ with round and strictly convex boundary. 

The first technique of \cite{Per1} is a gluing theorem for pRc metrics. It replaces the requirement that pRc metrics be standard on a collar (as is the case in the study of psc metrics), with an infinitesimal collar condition in terms of the second fundamental form. 

\begin{theorem}[\cite{Per1}]\label{GluingTheorem} Given two pRc Riemannian manifolds $(\Em_i^n,g_i)$ with an orientation reversing isometry $\Phi:\partial \Em^n_1 \rightarrow \partial \Em_2^n$ that satisfies $\2_1+\Phi^*\2_2$, there is a pRc metric $g$ on $\Em_1^n\cup_\Phi \Em_2^n$ that agrees with the $g_i$ outside of an arbitrarily small neighborhood of the gluing site. 
\end{theorem}

\noindent In the study of psc metrics on manifolds with boundary, it is typical to require the metric be a product on a collar of the boundary. This condition allows us to smoothly glue together two psc manifolds along an isometry of their boundaries. Clearly this collar condition cannot be used for pRc. Theorem \ref{GluingTheorem} motivates considering pRc Riemannian manifolds with strictly convex boundary, meaning that $\2$ is positive definite. Given two such manifolds with isometric boundary, Theorem \ref{GluingTheorem} allows us to smoothly glue them together. 

Another common theme in the study of psc metrics is the use of a psc concordance. A psc concordance is a psc metric on a cylinder $[0,1]\times \Em^n$ which splits as a product near each boundary component. When we referred to \cite[Theorem]{Gaj} in the proof of Observation \ref{Ahat} or Theorem \ref{SpinTheorem}, we were referring to a psc concordance. If we are interested at developing a similar tool for the study of pRc metrics, we must find a suitable boundary condition that replaces the product condition. As we have mentioned, the hypotheses of Theorem \ref{GluingTheorem} suggests convexity as an ideal boundary condition. Sadly, a pRc Riemannian manifold with weakly mean convex boundary must have a connected boundary by \cite[Theorem 1]{Law}. So if we are interested in a notion of pRc concordance, we must allow for one of the boundaries to have negative definite second fundamental form. Ideally though, $|\2|$ to be as small as possible. Motivated by this we make the following definition. Note the terminology we introduce here is meant as an analogy to ``almost non-negatively curved.'' 

\begin{definition}\label{AlmostWeak} Given a manifold $\Em^{n+1}$ with boundary, let $\En^n$ be a connected component of $\partial \Em^{n+1}$. We say that $(\En^{n},h)$ is an \emph{almost weakly convex} boundary component of $\Em^{n+1}$ if for each $\nu>0$, there is a metric $g_\nu$ on $\Em^{n+1}$ such that $g_\nu|_{\En^n} = h$ and $\2_{g_\nu}|_{\Em^n}>-\nu g_\nu$. Moreover, we say that $\Em^{n+1}$ is \emph{Ricci-positive with almost weakly convex boundary isometric to} $(\En^n,h)$ if each $g_\lambda$ is pRc. 
\end{definition}  

\noindent  The requirement that the boundary have a fixed isometry type in Definition \ref{AlmostWeak} is useful for our purposes as we always intend to use Theorem \ref{GluingTheorem} on such Riemannian manifolds. If one wishes to remove this requirement, the definition ought to require that the intrinsic diameter of $\En^n$ with respect to $g_\nu$ is bounded above uniformly. If not, then it is always possible to take $g_\nu= (1/\nu)^2 g$. 

The following observation is our main reason for introducing the concept of almost weakly convex boundary. 

\begin{obs}\label{WeCanGlue} Given a pRc metric on $\Em_1^n$ with strictly convex boundary and a pRc metric on $\Em_2^n$ with almost weakly convex boundary, if the boundaries are isometric then there is a pRc metric on $\Em_1^n\cup_\partial \Em_2^n$ by Theorem \ref{GluingTheorem}.
\end{obs}

With Definition \ref{AlmostWeak} in hand we can now give a reasonable candidate for a replacement of psc concordance suitable for the study of pRc metrics.

\begin{definition}\label{Concordance} We call a family of pRc metrics $G_\nu$ on $[0,1]\times \Em^n$ a \emph{pRc-concordance} from $g_1$ to $g_2$ if the boundary $\{0\}\times \Em^n$ is almost weakly convex and isometric to $(\Em^n,g_1)$ and the boundary $\{1\}\times \Em^n$ is strictly convex and isometric to $(\Em^n, R^2_\nu g_2)$ for some constant $R_\nu$. 
\end{definition}

\noindent Note that, unlike a psc-concordance, a pRc-concordance is a family of metrics and is directional by nature. Because it is directional it does not directly give an equivalence relation: reflexivity fails. We can however define an equivalence relation by requiring the existence of a pRc-concordance in both directions. One of the reasons we include the possibility of scaling $g_2$ by a constant, is so that a pRc metric is pRc-concordant to itself. Note that a feature of the Gauss equation is that neither $g_1$ nor $g_2$ necessarily are intrinsically pRc even if they are pRc-concordant.

 In the proof of Observation \ref{Ahat} and \ref{SpinTheorem} we referenced \cite[Theorem]{Gaj}. This claims that if two metrics are psc isotopic, then they are psc-concordant. The following result claims the same holds for pRc isotopies. 

\begin{theorem}\cite[Theorem 7]{BLBNew}\label{IsotopyConcordance} If $g_1$ and $g_2$ are connected via a path in $\pRc(\Em^n)$, then there is a pRc-concordance from $g_1$ to $g_2$. 
\end{theorem}
\noindent Note that by reversing the path, we also have a pRc-concordance from $g_2$ to $g_1$. The principal way in which we will use Theorem \ref{IsotopyConcordance} is summarized as follows.

\begin{theorem}\cite[Theorem C]{BLBNew}\label{FixBoundary} Given a pRc metric $g$ on $\Em^n$ with weakly convex boundary isometric to $g_1$, if $g_1$ and $g_2$ are pRc-isotopic, then there is a pRc metric $\tilde{g}$ on $\Em^n$ with strictly convex boundary isometric to $g_2$. 
\end{theorem}
\noindent The proof of Theorem \ref{FixBoundary} follows by performing a small conformal change so the boundary becomes strictly convex then by Observation \ref{WeCanGlue} we can glue the pRc-concordance at the almost weakly convex end to its boundary. After rescaling by an appropriate constant the resulting boundary will be strictly convex and isometric to $g_2$.

Theorem \ref{IsotopyConcordance} demonstrates that, unlike pRc Riemannian manifolds with weakly convex boundaries, pRc Riemannian manifolds with almost weakly convex boundaries may have disconnected boundary. In fact, the second technique introduced in \cite{Per1} is a construction of a pRc Riemannian manifold with almost weakly convex boundary that has \emph{arbitrarily many} connected components. 

\begin{theorem}\cite{Per1}\label{Docking} For $n\ge 4$ and any $k$, there is a family of pRc metrics on $\Sphere^n\setminus \left(\bigsqcup_k \Disk^m\right)$ with almost weakly convex boundary isometric to $\bigsqcup_k \Sphere^{m-1}_1$. 
\end{theorem}

\noindent For a summary of the construction of this metric we recommend the reader consult \cite[Section 3.1]{BLBThesis}. To finish the construction of pRc metrics on $\#_k\CP^2$, \cite{Per1} constructed a pRc metric on $\CP^2\setminus \Disk^4$ with round and convex boundary, named the \emph{core}. By Observation \ref{WeCanGlue}, we can glue $k$ copies of the core $\CP^2\setminus \Disk^4$ to $\Sphere^n\setminus \left( \bigsqcup_k \Disk^n\right)$ producing a pRc metric on $\#_k \CP^2$. 


Motivated by the construction in \cite{Per1}, we make the following definition. 
\begin{definition}\label{Core} We will say that $\Em^n$ admits a \emph{core metric}, if there is a pRc metric on $\Em^n\setminus \Disk^n$ with strictly convex boundary isometric to $\Sphere_1^{n-1}$. 
\end{definition}

\noindent Following \cite{Per1} it is immediate that you can form a pRc connected sum on $\#_i \Em_i^n$ if each $\Em_i^n$ admit a core metric. Other than $\CP^2$ in \cite{Per1}, $\Sphere^n$ was the only other space known to admit a core metric. The main goal of our previous work in \cite{BLBOld,BLBThesis,BLBNew} was in constructing new examples of core metrics, which we summarize now. 

\begin{theorem}\label{TheCores} The following manifolds admits core metrics: 
\begin{enumerate}
\item\cite[Theorem C]{BLBOld} $\CP^n$, $\HP^n$, and $\OP^2$;
\item\cite[Theorem B]{BLBNew} $\Sphere(\Ee)$, where $\Ee\rightarrow \Bee^n$ is a rank $k\ge 4$ vector bundle over a base $\Bee^n$ that admits a core metric;
\item\cite[Theorem C]{BLBThesis} the boundary of a tree-like plumbing of $\Disk^{k}$-bundles over $\Sphere^{k}$ for $k\ge 4$, or the boundary of plumbing a $\Disk^p$ and $\Disk^q$ bundle over spheres for $p\ge 3$ and $q\ge 4$.
\end{enumerate}
\end{theorem}

Motivated by the idea of almost weakly convex boundaries we make the following definition in analogy to a core metric.  

\begin{definition}\label{Socket} We will say that $\Em^n$ admits a \emph{socket metric}, if $\Em^n\setminus \Disk^n$ admits a pRc metric with almost weakly convex boundary isometric to $\Sphere_1^{n-1}$. 
\end{definition} 

\noindent Note that any core metric is a socket metric, but admitting a socket metric is weaker than admitting a core metric. By \cite[Thereom 1]{Law}, if $\Em^n$ admits a core metric, $\pi_1(\Em^n)=0$. However it is possible to concoct a socket metric on $\RP^n$ (remove two large geodesic balls from an appropriately sized $\Sphere_R^n$ and take the quotient under the antipodal map). With a little more work it is possible to show the following. 

\begin{theorem}\cite[Corollary 4.4]{BLBOld}\label{TheSockets} For $n\ge 4$, $\RP^n$ and lens spaces all admit socket metrics. 
\end{theorem}

We now summarize how Theorem \ref{Docking} may be used to mix-and-match core and socket metrics. 

\begin{corollary}\label{Connect} For $n\ge4$, if $\Em_i^n$ admits a core metric, then $\#_i\Em_i^n$ also admit a core metric, and in particular $\#_i\Em^n$ admits a pRc metric. 

If $\Ex^n$ admits a socket metric, then $\Ex^n\#\left(\#_i \Em_i^n \right)$ also admits a socket metric, and in particular admits a pRc metric. 
\end{corollary}

\begin{proof} We begin by claiming there is a family of pRc metrics on $\Disk^n\setminus \left(\bigsqcup_k \Disk^n\right)$ so that the large boundary is strictly convex and isometric to $\Sphere_{R_\nu}^{n-1}$ and the $k$ disjoint small boundaries are almost weakly convex and isometric to $\bigsqcup_k \Sphere_1^{n-1}$. Note that the boundary conditions just described are identical to the boundary conditions of a pRc-concordance, but the topology is not a cylinder. Hence one might call this a \emph{pRc-cobordism.} 

An analysis of the family of metrics constructed on $\Sphere^n\setminus \left(\bigsqcup_k \Disk^n\right)$ of Theorem \ref{Docking} in \cite{Per1}, shows that it is isometric to a positively curved doubly warped product metric outside of a neighborhood of the boundary. In particular it is isometric to $dt^2+ \cos^2(t) d\theta^2 + f^2(t)ds_{n-2}$, where $t\in [0, \pi/2]$ and $f(t)\in[0,R]$. Moreover the boundary components are all fixed along an arbitrarily small neighborhood of circle corresponding to the set $t=0$. Thus it is possible to find a large geodesic ball in $\Sphere^n$ equipped with this doubly warped product metric that corresponds to the subset $\Disk^n\setminus \left(\bigsqcup_k \Disk^n\right)$ of $\Sphere^n\setminus \left(\bigsqcup_k \Disk^n\right)$ equipped with the family of metrics in Theorem \ref{Docking}, so that the large boundary of $\Disk^n\setminus \left(\bigsqcup_k \Disk^n\right)$ is isometric to the boundary of the geodesic ball. 

For dimension reasons, the boundary of the geodesic ball is isometric to a warped product $ds^2+ p^2(s)ds_{n-2}^2$ with $s\in[0,S]$. It is elementary to check that $p(s)$ is a concave down function, and hence the boundary has instrinsic positive curvature. It follows that the second fundamental form is positive definite. Note that the metric
$$ds^2 + \left((1-\lambda) p(s) +     \lambda S \sin(\pi s/S) \right)^2 ds_{n-2}^2,$$
is also positively curved for each $\lambda\in[0,1]$ and hence provides a pRc-isotopy from the boundary metric of this geodesic ball to a round metric. We may therefore apply Theorem \ref{FixBoundary} to produce a family of pRc metrics on $\Disk^n\setminus \left(\bigsqcup_k \Disk^n\right)$ that satisfy the desired boundary conditions. 

By Observation \ref{WeCanGlue} we can attach $\bigsqcup_{i=1}^k \left(\Em_i^n\setminus \Disk^n\right)$ with the core metrics to the almost weakly convex boundary of $\Disk^n\setminus \left(\bigsqcup_k \Disk^n\right)$. After rescaling by an appropriate constant, we have produced a core metric for $\#_i \Em_i^n$. 

In order to produce a socket metric on $\Ex^n\# \left(\#_i\Em^n_i\right)$ we repeat the above construction, but leave one of the boundary components of $\Disk^n\setminus \left(\bigsqcup_k \Disk^n\right)$ free. Then using Observation \ref{WeCanGlue} we attach the large boundary to $\Ex^n\setminus \Disk^n$. The result is a socket metric for $\Ex^n\#\left(\#_i \Em_i^n\right).$ 
\end{proof}

We did not claim Corollary \ref{Connect} in \cite{BLBNew}, because we did not think to apply Theorem \ref{FixBoundary} to the metrics of Theorem \ref{Docking}. With Corollary \ref{Connect} we can now explain how Theorems \ref{A} and \ref{ModuliList} follows from Theorems \ref{B} and \ref{SocketModuli} respectively. 

\begin{proof}[Proof of Theorem \ref{A}] By (1) of Theorem \ref{TheCores} $\Sphere^n$, $\CP^n$, $\HP^n$, and $\OP^2$ all admit core metric. By (2) of Theorem \ref{TheCores} we may construct a core metric on the total space of a linear $\Sphere^m$-bundle over a Riemannian manifold admitting a core metric, and by Corollary \ref{Connect} we may construct a core metric on the connected sum of any manifolds admitting a core metric. It follows that every manifold in $\Class$ of Definition \ref{ClassDefinition} admits a core metric. Note that a core metric is also a socket metric, hence the claim follows from Theorem \ref{B} for such manifolds. 

By Theorem \ref{TheSockets} $L^{4k-1}$ admits a socket metric for any lens space. As any $\Em^{4k-1}\in \Class$ admits a core metric, by Corollary \ref{Connect} there is a socket metric on $\L^{4k-1}\#\Em^{4k-1}$, hence the claim follows from Theorem \ref{B} for such manifolds. 
\end{proof}

\begin{proof}[Proof of Theorem \ref{ModuliList}] All of the manifolds listed in Theorem \ref{ModuliList} have socket metrics (as explained in the proof of Theorem \ref{A}). So by Theorem \ref{SocketModuli} it suffices to show that the manifolds listed satisfy the topological hypotheses. Note that products and connected sums preserves stable parallelizability, so any of the iterated products and connected sums of spheres is stably parallelizable and hence has vanishing Pontryagin classes. 

Note that for an odd $m$, $H^1(L(m;q_1,\dots, q_{2k}),\mathbf{Z}/2\mathbf{Z})=0$. If $a$ is a generator of $H^2(L(m;q_1,\dots, q_{2k}),\mathbf{Z})$ then the total Pontryagin class can be computed (see \cite[Corollary 3.2]{Szczar})
$$p(L(m;q_1,\dots, q_{2k})) = (1+q_1^2a^2)\cdots (1+q_{2k}^2a^2)  .$$
The requirement that all Pontryagin classes vanish is precisely that $p(L(m;q_1,\dots, q_{2k})) = 1$. Since $H^{4i}(L(m;q_1,\dots, q_{2k}),\mathbf{Z})=\mathbf{Z}/m\mathbf{Z}$, the coefficients of $a^{2i}$ in $p(L(m;q_1,\dots, q_{2k}))$ is precisely $e_{i}(q_1^2,\dots, q_{2k}^2) \mod m$. 
\end{proof}


\subsection{The new work}\label{SurgeryCores}


In the introduction, we explained that a manifold admits a socket metric if it is pRc and almost admits a round hemisphere. One could similarly describe a core metrics as a pRc metric that admits a round hemisphere. By Theorem \ref{GluingTheorem} a manifold that admits a socket metric, can be glued to $\Disk_1^n(\pi/2-\nu)$ for any $\nu>0$. Hence it ``almost admits a hemisphere,'' and the converse is obvious. Similarly if a manifold admits a core metric it can be glued to $\Disk_1^n(\pi/2)$ using Theorem \ref{GluingTheorem}. Morally, this means that manifolds that admit socket metrics behave like the round metric on a nearly global scale. Hence it is plausible that the comment about the sphere following the proof of Theorem \ref{surgerytheorem} may apply equally well to any spin manifold that admits a socket metric. This is precisely the idea that allows us to prove Theorem \ref{B} from the following lemma.  

\begin{lemma}\label{DiskLemma} For all $n\ge2$ and $m\ge 4$, given any $\rho>0$ there is a core metric on $\Sphere^{n+m}$ that admits an isometric embedding $\iota:\Sphere^{n}_\rho\times \Disk^m_1(\pi/8)\hookrightarrow \Disk^{n+m}$. 
\end{lemma}

Assuming Lemma \ref{DiskLemma} we can complete the proofs of our main theorems. 

\begin{proof}[Proof of Theorem \ref{B}] As explained above in Section \ref{pRcsums}, it is always possible to glue together a socket metric and a core metric. If $\Em^{4k-1}$ admits a socket metric, then by attaching the core metric of Lemma \ref{DiskLemma} it is possible to find an embedding as in Theorem \ref{surgerytheorem} for any $\rho>0$. The claim follows. 
\end{proof}

\begin{proof}[Proof of Theorem \ref{SocketModuli}] As explained above in Section \ref{pRcsums}, it is always possible to glue together a socket metric and a core metric. If $\Em^{4k-1}$ admits a socket metric, then by attaching the core metric of Lemma \ref{DiskLemma} it is possible to find an embedding as in Theorem \ref{SpinModuli} for any $\rho>0$. The claim follows. 
\end{proof}

In order to prove Lemma \ref{DiskLemma} we will consider the family of metrics constructed on $\Sphere^{n+m}$ discussed following the proof of Theorem \ref{surgerytheorem} using the work of \cite{SY}. Specifically we must find an embedding of $\Disk^{n+m}$ that contains the desired embedding on its interior and has convex boundary. This is essentially obvious once one understands the construction in \cite{SY}. While this much is straightforward, the metric so constructed restricted to the boundary is not round. Luckily the metric restricted to the boundary $\Sphere^{n+m-1}$ is essentially the same as the metric constructed $\Sphere^{n+m}$, and so will be intrinsically pRc. Using Theorem \ref{FixBoundary} it is possible to deform the boundary to become round provided it is connected via a path of pRc metrics to the round metric. The following Lemma claims that this is possible for the particular sort of metric constructed on $\Sphere^{n+m-1}$ in \cite{SY}. 

\begin{lemma}\label{paths} Let $k=dr^2+ h^2(r)ds_n^2 + f^2(r) ds_m^2$ with $r\in[0,R]$ be a pRc doubly warped product metric on $\Sphere^{n+m+1}$. Suppose that there is an $0<R_1<R_2<R$ such that $f''(r)>0$ for $r<R_1$ and $f''(r)<0$ for $r>R_1$ and $h''(r)<0$ for $r<R_2$ and $h(r)\equiv \rho$ for $r\ge R_2$. Then $k$ is connected via a path of pRc metrics to a round metric. 
\end{lemma}
\begin{proof} Let $f_1(r)$ be any function defined on $[0,R]$ satisfying the same boundary conditions as $f(r)$ (see \cite[Section 1.4.5]{Pet}), suppose moreover that $f_1''(r)<0$ and $f_1(R_1)=f(R_1)$. Let $f_\lambda(r)$ denote the convex combination of $f_0(r)$ and $f_1(r)$ for $\lambda\in [0,1]$. Let $k_\lambda=dr^2+h^2(r) ds_n^2  + f_\lambda^2(r)ds_m^2$. We claim that $g_\lambda$ has pRc for each $\lambda\in[0,1]$. Let $\partial_r$, $\theta_i\in T\Sphere^n$ and $\phi_i\in T\Sphere^m$ be coordinate frame for $g_\lambda$. From \cite[4.2.4]{Pet}, we have
\begin{align}
\label{rRicci} Ric_{k_\lambda}(\partial_r,\partial_r) & = -n\frac{\ddot{h}(r)}{h(r)} -m\frac{\ddot{f}_\lambda(r)}{f_\lambda(r)}.\\
\label{hRicci} \Ric_{k_\lambda}(\theta_1,\theta_2) & = (n-1) \frac{1- \dot{h}^2(r)}{h^2(r)} -  \frac{\ddot{h}(r)}{h(r)} -m\frac{\dot{f}_\lambda(r)\dot{h}(r)}{f_\lambda(r)h(r)}\\
\label{fRicci} \Ric_{k_\lambda}(\phi_1,\phi_2) & = (m-1) \frac{1- \dot{f}_\lambda^2(r)}{f_\lambda^2(r)} - \frac{\ddot{f}_\lambda(r)}{f_\lambda(r)} -n\frac{\dot{f}_\lambda(r)\dot{h}(r)}{f_\lambda(r)h(r)}.
\end{align}

Let us consider first (\ref{rRicci}). For $r>R_1$, $h''(r)\le 0$ and $f_\lambda''(r)<0$ hence (\ref{rRicci}) is positive for all $\lambda\in[0,1]$. For $r\le R_2$, $h''(r)$ is constant with respect to $\lambda$ and $f_\lambda''(r)$ is decreasing with respect to $\lambda$, hence (\ref{rRicci}) is increasing with respect to $\lambda$. Since (\ref{rRicci}) is assumed to be positive at $\lambda=0$, we conclude that (\ref{rRicci}) is positive for all $\lambda\in[0,1]$. 

Let us consider next (\ref{hRicci}). For $r\ge R_2$, we have that (\ref{hRicci}) is given by $(n-1)/\rho^2$ and is hence positive regardless of $\lambda$. For $r\le R_1$, let us consider (\ref{hRicci}) when $\lambda=0$, which by assumption is positive. Note that $f_\lambda'(r)\ge 0$ and $h'(r)>0$, hence $\dot{f}_\lambda(r)\dot{h}(r)>0$. It follows that the first two terms in (\ref{hRicci}) dominate the last term. Note that $\dot{f}\lambda(r)$ is decreasing while $f_\lambda(r)$ is increasing with respect to $\lambda$, hence $\dot{f}_\lambda(r)/f_\lambda(r)$ is decreasing. It follows that (\ref{hRicci}) is increasing with respect to $\lambda$ and hence is positive for all $\lambda\in[0,1]$. For $R_1<r<R_2$, let us again consider $\lambda = 0$. In this case $\dot{h}(r)>0$, but it is possible that $\dot{f}(r)$ changes sign at a single point $r_0$ from positive to negative. For $r\ge r_0$, the last term in (\ref{hRicci}) is nonnegative. Since $\dot{f}_1(r)<0$, it follows that this term is nonnegative for all $\lambda\in[0,1]$, and hence (\ref{hRicci}) is positive for all $\lambda\in [0,1]$. For $R_1<r<r_0$, the argument proceeds identically to the case of $r\le R_1$. 

Finally, let us consider (\ref{fRicci}). Let us again begin by considering $\lambda=0$. Note that there is exactly one $r_0>R_1$ for which $\dot{f}_0(r)$ changes sign from positive to negative. For $r\le R_1$, note that the last two terms in (\ref{fRicci}) are both nonpositive. For $R_1< r \le r_0$ the last term is negative and the middle term is positive. For $r>r_0$, the second two terms are positive. Because (\ref{fRicci}) is assumed to be positive, we deduce that $\dot{f}_0(r)<1$ for $r\le R_1$. Because $\ddot{f}_0(r)<0$ for $r>R_1$, and $\dot{f}_0(R)=-1$ we deduce that $-1\le \dot{f}_0(r)<1$. For a fixed $r>r_0$, we have that $\dot{f}_\lambda(r)$ is decreasing and hence $\dot{f}(r)^2\le 1$ for all $\lambda$. As the other two terms in (\ref{fRicci}) are positive, it follows that (\ref{fRicci}) is positive for all $\lambda\in[0,1]$. For a fixed $r\le r_0$, $\dot{f}^2(r)$ will decrease initially until $\dot{f}_\lambda(r)=0$ for some $\lambda$. At this same $\lambda$, the last two terms in (\ref{fRicci}) become positive. Hence (\ref{fRicci}) will become positive beyond this $\lambda$. Before this $\lambda$ we note that the first term in (\ref{fRicci}) was increasing, while the second two were decreasing with respect to $\lambda$. 

We therefore have shown that $k_0$ and $k_1$ are connected via a path of pRc metrics. It remains to check that $k_1$ is connected via a path of pRc metrics to a round metric. If we let $h_2(r)= R \sin (r/R)$ and $f_2(r) =R \cos (r/R)$, define $h_\lambda$ and $f_\lambda$ as the convex combination of $h(r)$ with $h_2(r)$ and $f_1(r)$ with $f_2(r)$ for $\lambda\in[1,2]$, and let $k_\lambda = dr^2 + h_\lambda^2(r) ds_n^2+f_\lambda^2(r)ds_m^2$. It is elementary to check that (\ref{rRicci}), (\ref{hRicci}), and (\ref{fRicci}) are positive as $h_\lambda(r)$ and $f_\lambda(r)$ are strictly concave down for all $\lambda>1$. 
\end{proof}

We are now ready to finish our proof of Lemma \ref{DiskLemma}. 

\begin{proof}[Proof of Lemma \ref{DiskLemma}]  For any $\rho>0$, consider the pRc Riemannian manifold $\Sphere_\rho^n\times \Sphere_1^m$. Define two embeddings $\iota,\iota': \Sphere_\rho^n\times \Disk_1^m(\pi/8) \hookrightarrow \Sphere_\rho^n\times \Sphere_1^m$ given by the identity on the first factor and by the exponential map centered at two points $x_0$ and $x_1$ a distance of $\pi/2$ apart in $\Sphere_1^m$. By \cite[Lemma 1]{SY}, we may for $\rho$ sufficiently small, perform a surgery on $\iota'$ to produce a pRc metric on
$$\left[ \left(\Sphere^n\times \Sphere^m \right)\setminus \iota'\left(\Sphere^n\times \Disk^m\right) \right] \cup_\iota \left[\Disk^{n+1} \times \Sphere^{m-1}\right]=\Sphere^{n+m}.$$
Moreover this metric takes the form $k:=dr^2 + h^2(r) ds_{n}^2 + f^2(r)ds_{m-1}^2$, where $r\in[0, R]$ agrees with the radial coordinate of $\Sphere^m\setminus \Disk^m$ in this first factor and the radial coordinate of $\Disk^{n+1}$ in the second factor. Moreover the interval can be subdivided into $0<R_1<R_2<R_3<R$ such that
\begin{enumerate}
\item\label{inflection} $f(r)$ has a unique inflection point at $r=R_1$, is strictly concave up for $r<R_1$ and strictly concave down for $r>R_1$. 
\item\label{isconcave} $h(r)$ is strictly concave down for $r<R_2$ and $h(r)= \rho$ for $r\ge R_2$.
\item\label{isround} $f(r) =  \cos(r-R_3+ \pi/8)$ for $r\ge R_3$.  
\end{enumerate}
In addition to all this, there is still an isometric embedding $\iota:\Sphere_\rho^n\times \Disk_1^m(\pi/8)\hookrightarrow \Sphere^{n+m}$. Our first goal is to identify an embedding $\Disk^{n+m}\hookrightarrow \Sphere^{n+m}$ that contains the image of $\iota$ in its interior and that has weakly convex boundary. 

Implicit to the definition of the metric $k$, we are identifying $\Sphere^{n+m}$ as a quotient of $[0,R]\times \Sphere^n\times \Sphere^{m-1}$. If we further decompose $\Sphere^{m-1}=[0,\pi]\times \Sphere^{m-2}$, then the metric splits as
 $$k= dt^2+ h^2(t)ds_n^2 + f^2(t)dx^2 + \cos^2( x) f^2(t) ds_{n-2}^2.$$
 Given a function $\xi:[0,R]\rightarrow [0,\pi]$, we define the subset 
 $$\Disk(\xi):=\{(r,p,x,q)\in[0,R]\times \Sphere^n\times [0,\pi]\times \Sphere^{m-2}: x\le \xi(r)\}\subseteq \Sphere^{n+m}.$$ 
 Intersecting $\Disk(\xi)$ with the two handles gives
 $$\Disk(\xi) =\left[ \Sphere^n\times \Disk_+^m\right] \cup_{\Sphere^n\times \Disk^{m-1}} \left[\Disk^{n+1}\times \Disk^{m-1} \right].$$
 From this it is clear that $\Disk(\xi)$ is homeomorphic to $\Disk^{n+m}$. We will choose $\xi(r)$ so that $\Disk(\xi)$ transparently has smooth boundary. 
 
For $r<R-\pi/4$, we choose $\xi(r)=c$ to be constant. Recall that $[R_3,R]\times \Sphere^n\times \Sphere^{m-1}$ equipped with $k$ is isometric to $\Sphere_\rho^n\times \left(\Sphere_1^m\setminus \Disk_1^m(\pi/8)\right)$. For $r>R-\pi/8$ we want to choose $\xi(r)$ so that this portion of $D(\xi)$ agrees with a geodesic ball of radius $\pi/4$ centered at $x_1$. 
We will choose $c$ sufficiently small, so that it is possible to extend $\xi(r)$ smoothly so that $\xi(r)$ is concave down for all $r>R-\pi/4$. With this choice of $\xi(r)$, the boundary of $D(\xi)$ is clearly smooth. 

What remains is to check that the boundary is weakly convex. For $r< R- \pi/4$, as $\xi(r)=c$ we have that the inward unit normal of $\Sphere(\xi):=\partial \Disk(\xi)$ is given by $-\partial_x/f(r)$ and a frame for $T\Sphere(\xi)$ is spanned by $\partial_t$, $\theta_i\in T\Sphere^n$, and $\phi_i\in T\Sphere^{m-2}$. In this frame all principal curvatures are zero other than $\2(\phi,\phi)= \cot c /f(r)$. For $r>R-\pi/4$, it is routine to check that the principal curvatures tangent to $\Sphere^{m-1}$ are positive if $\xi(r)$ is concave down. 

We claim that the metric $k$ restricted to $\Sphere(\xi)=\Sphere^{n+m-1}$ takes the form 
$$k|_{\Sphere(\xi)}=ds^2+ h^2(s)ds_n^2 + p^2(s)ds_{n-2}^2.$$
Where $s\in [0,R']$ where $p(s)= \cos(c)f(s)$ for $s<R-\pi/4$ and $p''(s)<0$ for all $\ge R_3$. Moreover, $k|_{\Sphere(\xi)}$ has pRc. Because $\xi(r)=c$ for $r<R-\pi/4$, it is clear that $k|_{\Sphere(\xi)}$ takes this form. That $k|_{\Sphere(\xi)}$ has pRc follows from the fact that scaling $f(r)$ by a constant less than $1$ in a doubly warped product metric has the result of increasing or leaving fixed all sectional curvatures. For $r\ge R-\pi/4$, we have that
$$k|_{\Sphere(\xi)} = (1+\cos^2(\xi(r))\dot{\xi}^2(r))dr^2+ h^2(r) ds_n^2 + \cos^2(\xi(r))f^2(r) ds_{m-2}^2 = ds^2 +h^2(r) ds_n^2 + p^2(s) ds_{m-2}^2.$$
We have explained that $\Sphere(\xi)\cap \Sphere^{m-1}$ is strictly convex and hence has positive sectional curvature. It follows that $p''(s)<0$ as claimed. As $h(s)=\rho$ for such $s$, it is easy to verify that the metric has pRc. 

By Lemma \ref{paths}, $k|_{\Sphere(\xi)}$ is connected via a path of pRc metrics to a round metric. By Theorem \ref{FixBoundary}, there is a pRc metric on $\Disk^{n+m}$ with round and convex boundary that still admits the desired embedding. This construction works for any $\rho$, as desired. 
\end{proof}


\section{Further Applications}\label{AppsSection}



\subsection{The space of core metrics}\label{CoresSection}


For a smooth manifold $\Em^n$ fix $\Disk^n\subseteq \Em^n$. We define $\Cores(\Em^n)\subseteq \pRc(\Em^n\setminus \Disk^n)$ to be the space of core metrics for $\Em^n$. Note that the action of diffeomorphisms of $\Em^n\setminus \Disk^n$ on $\pRc(\Em^n\setminus \Disk^n)$ leaves on $\Cores(\Em^n)$ fixed, and hence we may define the moduli space $\MCores(\Em^n)$. Much of our own work has been in showing that $\Cores(\Em^n)$ is nonempty for manifolds where $\pRc(\Em^n)$ is known to be nonempty. The end goal of this to show that $\pRc(\#_k \Em^n )$ is nonempty. In this section we hope to convince the reader that $\Cores(\Em^n)$ is an interesting space in its own right. 

In Corollary \ref{Connect} we described a way of combining core metrics on $\Em_i^n$ to a core metric on $\#_i \Em_i^n$. When we restrict ourselves to core metrics on $\Sphere^n$, we note that this procedure gives a way to take two core metrics on $\Sphere^n$ and produce a new core metric on $\Sphere^n$. In other words there is a binary operation
$$\mu: \Cores(\Sphere^n)\times \Cores(\Sphere^n) \rightarrow \Cores(\Sphere^n).$$
While there were many choices hidden in the proof of Corollary \ref{Connect}, we believe it is possible to make $\mu$ a well-defined, continuous product. This construction is morally similar to the construction of a product structure on $\psc(\Em^n)$ in \cite{Walsh} and more recently on the space of all metrics of $k$-pRc on $\Sphere^n$ in \cite{WW}. In both cases it was shown that this product structure comes from an $n$-fold loop space structure. We conjecture that the product structure of $\Cores(\Sphere^n)$ behaves similarly. 

\begin{conj}\label{Conjecture} $\Cores(\Sphere^n)$ equipped with $\mu$ is an associative $H$-space and is homotopy equivalent to a loop space. 
\end{conj}

In \cite{Walsh} the product structure is essentially given by the psc connected sum of \cite{GL}, and so the product is defined for any psc metric on the sphere. There is no general way to form a pRc connected sum. To our knowledge Corollary \ref{Connect} is the most general technique for forming pRc connected sums. And as it requires the presupposition of core metrics, Conjecture \ref{Conjecture} represents a natural extension of the ideas of \cite{Walsh} to the study of pRc metrics.  
 
While we have no way of constructing a product structure on the space $\pRc(\Sphere^n)$, one plausible way to approach this is to study the following embedding. 
\begin{prop} There is an embedding $\kap:\Cores(\Em^n)\hookrightarrow \pRc(\Em^n)$. 
\end{prop}
\begin{proof} We define a function $\nu:\Cores(\Em^n)\rightarrow \mathbf{R}_+$, for $g\in \Cores(\Em^n)$ let
$$ \nu(g):=  \dfrac{1}{2} \cdot \inf \left\{ \2_{g}(E,E)  :  {E\in T\Sphere^{n-1} } \text{ and }  g(E,E)=1 \right\}.$$
This function is continuous. Fix a one-parameter family of pRc metrics $h(\nu)$ on $\Disk^n\subseteq \Em^n $ so that the boundary is isometric to $\Sphere_1^{n-1}$ and $\2_{g(\nu)}=-\nu g(\nu)$. For instance one can take $h(\nu)$ to be the pull-back of an appropriately sized round metric restricted to a geodesic ball of an appropriate radius. 

As described in \cite[Theorem 2]{BWW}, it is possible to modify Theorem \ref{GluingTheorem} so that the pRc metric $g$ produced on $\Em_1^n\cup_\partial \Em_2^n$ depends continuously on the two metrics $g_1$ and $g_2$ and the isometry of the boundary. Thus we can define $\kap(g)\in\pRc(\Em^n)$ to be the result of applying \cite[Theorem 2]{BWW} to the metric $g$ on $\Em^n\setminus \Disk^n$ and $h(\nu(g))$ on $\Disk^n$ identified along the fixed embedding.  
\end{proof} 

To our knowledge the space $\Cores(\Em^n)$ has never been directly studied. The closest relative to $\Cores(\Em^n)$ that has been studied is the space of nonnegative Ricci curvature metrics on $\Disk^3$ with strictly convex boundaries. This space was studied in \cite{AMW}, where it was shown that it is path connected and the associated moduli space is contractible. The proofs of these two facts apply equally well to $\Cores(\Sphere^3)$ and $\MCores(\Sphere^3)$. Beyond this what is known about the topology of $\Cores(\Em^n)$ is what can be inherited from $\pRc(\Em^n)$ under the embedding $\kap$. Take for example the work of \cite{CSS}. Consider the diffeomorphisms of the disk leaving a neighborhood of the boundary fixed $\Diff(\Disk^n,\partial)$. Given an inclusion $\Disk^n\subseteq \Em^n$ there is a corresponding inclusion $\Diff(\Disk^n,\partial)\subseteq \Diff(\Em^n)$. The main results of \cite{CS,CSS} claim that many higher homotopy groups of $\pRc(\Em^n)$ are nontrivial for $n\ge 7$. These nontrivial elements are inherited from the action of $\Diff(\Disk^n,\partial)$ on $\pRc(\Em^n)$. Note that if we fix two disjoint disks in $\Em^n$, one disk for the action of $\Diff(\Disk^n,\partial)$ and the other for the inclusion $\kap:\Cores(\Em^n)\hookrightarrow \pRc(\Em^n)$, then the action of $\Diff(\Disk^n,\partial)$ will leave the image of $\kap$ invariant. It follows from \cite[Corollary 1.9]{CSS} that for $n\ge 7$, $\Cores(\Em^n)$ will have infinitely many nontrivial higher homotopy groups provided that $\Em^n$ is spin manifold that admits a core metric. 

This discussion shows that Conjecture \ref{Conjecture} is not trivial, as the space $\Cores(\Sphere^n)$ is not contractible. We conclude this section by noting that Lemma \ref{DiskLemma} combined with Theorem \ref{surgerytheorem} and Theorem \ref{SpinModuli} show that we can deduce the non-triviality of the space of all core metrics on a spin manifold.  

\begin{corollary}\label{CoreSpace} For $k\ge 2$, let $\Em^{4k-1}$ be any spin manifold that admits a core metric, then $\Cores(\Em^{4k-1})$ has infinitely many path components. If moreover $\Em^{4k-1}$ satisfies the topological hypotheses of Theorem \ref{KSTheorem} then $\MCores(\Em^{4k-1})$ has infinitely many connected components. 
\end{corollary}


\subsection{Exotic smooth structures}\label{ExoticSection}


Note that Theorem \ref{WraithSpaceTheorem} and \ref{WraithModuliTheorem} are both stated for any smooth homotopy sphere in $bP_{4k}$. This is because any element of $bP_{4k}$ is diffeomorphic to $\Sphere^{4k-1}\#\left( \#_p \Sigma^{4k-1}\right)$ for some $p$. Note that both Theorem \ref{B} and Theorem \ref{SocketModuli} apply equally well to $\Em^{4k-1}\# \left( \#_p \Sigma^{4k-1}\right)$, and it is natural to ask if these represent distinct smooth manifolds for each $p$. 

 For any topological manifold $\Em^n$ that is not necessarily a homotopy spheres, we define the set $\Structure(\Em^n)$ as the set of diffeomorphism classes of smooth manifolds that are homeomorphic to $\Em^n$. This set can also be thought of the collection of all ``smoothings'' of the underlying topological manifold. For $n\ge 5$,  $\Structure(\Sphere^n)=\Theta_n$ has a group structure given by connected sum. For any manifold $\Em^n$, we see that $\Theta_n$ acts on $\Structure(\Em^n)$ by connected sum. Assuming that $\Em^n$ is already smooth, we can consider the subgroup $\Inertia(\Em^n)\le \Theta_n$ that leaves the class $[\Em^n]\in \Structure(\Em^n)$ fixed, in other words $\Sigma^n\in \Inertia(\Em^n)$ precisely when $\Em^n\#\Sigma^n$ is diffeomorphic to $\Em^n$. Thus for $\Sigma^{n}\notin \Inertia(\Em^n)$, $\Em^n\# \left(\# \Sigma^n\right)$ is a distinct smooth manifold. While computing $\Inertia(\Em^n)$ is beyond the scope of this note, we would like to record from the literature what is known about the inertia groups found in Theorem \ref{A} above. 
 
 Before we continue with our discussion of inertia groups, we remark that Lemma \ref{DiskLemma} combined with the work of \cite{WraithExotic} allow us to construct connected sums with a few more exotic spheres than those in $bP_{4k}$. 
 
 \begin{corollary}\label{AllTheSpheres} If $\Em^n$ admits a socket metric, then $\Em^n\# \Sigma^n$ also admits a socket metric and hence a pRc metric whenever:
 \begin{enumerate}
 \item $n=4k-1$, and $\Sigma\in bP_{4k}\cong\mathbf{Z}/b_k\mathbf{Z}$,
 \item $n=4k+1$, and $\Sigma\in bP_{4k+2} = \mathbf{Z}/2\mathbf{Z}$ whenever $k\notin \{1,3,7,15,31\}$,
 \item $n=8,16$, and $\Sigma \in \Theta_n=\mathbf{Z}/2\mathbf{Z}$,
 \item $n=19$, and $\Sigma \in \Theta_{19}=\mathbf{Z}/2\mathbf{Z}\oplus \mathbf{Z}/b_5\mathbf{Z}$.
 \end{enumerate}
 \end{corollary} 
 \begin{proof} By Lemma \ref{DiskLemma}, any manifold that admits a socket metric admits a family of pRc metrics so that there is an isometric embedding $\iota: \Sphere^n_\rho \times \Disk_1^m(\pi/8)\hookrightarrow \Em^{n+m}$ whenever $n\ge 2$ and $m\ge 4$.
 
 When $m=k$ and $n=k-1$, we may use \cite[Lemma 2.4]{WraithExotic} to perform iterated framed surgeries on $\Sphere^{k-1}$ starting from $\iota$, thus it is possible to achieve any element of $bP_{2k}$ as explained in the proof of \cite[Theorem 2.2]{WraithExotic}.
 
When $m>n+1$, we may again use \cite[Lemma 2.4]{WraithExotic} to perform surgery on $\iota$. By \cite[Satz 12.1]{StolzBook}, the remaining spurious examples in $\Theta_8$, $\Theta_{16}$, and $\Theta_{19}$ can all be achieved in this way as explained in the proof of \cite[Theorem 2.3]{WraithExotic}.
 \end{proof} 
 
 With Corollary \ref{AllTheSpheres} in hand, we can now list what is known about the inertia groups of manifolds found in Theorem \ref{A}. 
 
 \begin{theorem}\label{InertiaTheorem} 
 \begin{enumerate}
 \item If $\Sigma^{2n-1}\in \Inertia(L(m;q_1,\dots q_n))$, then $\#_m \Sigma^{2n-1} \cong \Sphere^{2n-1}$. 
 \item\cite[Theorem 1]{Kaw} $\Inertia(\CP^n)=0$ for $n\le 8$.
 \item\cite{Brow} $\Inertia(\RP^{4k-1})\cap bP_{4k-1}=0$ for all $k>1$. 
 \item\cite[Theorem A]{Schultz} for $n\ge 5$, if $\Em^n$ is a product of spheres, then $\Inertia(\Em^n)=0$.
 \item\cite[Corollary 2]{Frame} for $n\ge 7$, if $\Em^n$ is $2$-connected then $\Inertia(\Em^n)=\Inertia(\Em^n\# \left(\#_p \Sphere^2\times \Sphere^{n-2}\right))$. 
 \end{enumerate}
 \end{theorem}

Sadly $\Inertia(\Em^n)$ is not well behaved with respect to connected sums or with products, so we are limited to only a few special cases of manifolds produced in Theorem \ref{A}. The following is immediate from Theorem \ref{TheCores}, Corollary \ref{AllTheSpheres}, and Theorem \ref{InertiaTheorem}

\begin{corollary}\label{ExoticRicci} In each of the following, $\Em^n\#\Sigma^n$ admits a metric of pRc and each of the $\Em^n\#\Sigma^n$ represent distinct smooth structures on $\Em^n$ for distinct $\Sigma^n$. 
\begin{enumerate}
\item For any $k>1$ and any odd $m$, $\Em^{4k+1}=L(m;q_1,\dots,q_{2k+1})$ with $\Sigma^{4k+1}\in bP_{4k+2}$. 
\item For each $k>1$ and any prime $p>b_k$, $\Em^{4k-1}=L(p;q_1,\dots, q_{2k})$ with $\Sigma^{4k-1}\in bP_{4k+2}$. 
\item For each $k>1$, $\Em^{4k-1}=\RP^{4k-1}$ with $\Sigma^{4k-1}\in bP_{4k}$.
\item For $k=4,8$, $\Em^{2k}=\CP^k$ with $\Sigma^{2k}\in \Theta_{2k}$
\item For $k>3$, $\Em^{2k-1}$ defined as follows with $\Sigma^{2k-1}\in bP_{2k}$. For any multi-index $I$ such that $I_1+\dots + I_l=2k-1$ such that $I_i\ge 3$ and any $p$ let
$$\Em^{2k-1}:= \left(\prod_{1\le i\le l} \Sphere^{I_i}\right) \# \left(\#_p \left( \Sphere^2\times \Sphere^{2k-3}\right)\right).$$
\end{enumerate}
\end{corollary}

We end by noting that each of the examples $\Em^n$ produced in Corollary \ref{ExoticRicci} in dimension $n=4k-1$ will also satisfy the hypotheses of Theorem \ref{B} (and in fact Corollary \ref{SpinModuli}), and hence we can enhance the conclusion as follows.

\begin{corollary} Let $\Em^{4k-1}$ be any of the manifolds listed in (2) and (5) of Corollary \ref{ExoticRicci}, then $\MpRc(\Em^{4k-1})$ and hence $\pRc(\Em^{4k-1})$ have infinitely many path components. 
\end{corollary}

\bibliographystyle{alpha.bst} 
\bibliography{references}

\end{document}